\documentclass[reqno,9pt,oneside]{amsart}
\usepackage{geometry}
\usepackage{fancybox,ascmac,comment}
\usepackage{amssymb,mathrsfs}
\usepackage{amsfonts}
\usepackage{amsmath}
\usepackage{amsthm}
\usepackage{mathtools}
\usepackage{enumerate}

\usepackage{mleftright}
\mleftright
\mathtoolsset{showonlyrefs=true}
\numberwithin{equation}{section}

\newcommand{\ep}{\varepsilon}
\newcommand{\bk}{\mathbf{k}}
\newcommand{\bm}{\mathbf{m}}
\newcommand{\bbZ}{\mathbb{Z}}
\renewcommand{\Re}{\mathrm{Re}}

\theoremstyle{definition}
\newtheorem{theorem}{Theorem}[section]
\newtheorem{proposition}[theorem]{Proposition}
\newtheorem{lemma}[theorem]{Lemma}
\newtheorem{corollary}[theorem]{Corollary}

\newtheorem{definition}[theorem]{Definition}
\theoremstyle{remark}
\newtheorem{remark}[theorem]{Remark}

\allowdisplaybreaks
\title{Cyclic sum formula for certain parametrized multiple zeta values}
\author{Hanamichi Kawamura, Anju Yokoi}
\address[Hanamichi Kawamura]{Department of Mathematics, Faculty of Science Division I, Tokyo University of Science, 1-3 Kagurazaka, Shinjuku-ku, Tokyo, 162-8601, Japan}
\email{1121026@ed.tus.ac.jp}
\address[Anju Yokoi]{Ikeda Senior High School Attached to Osaka Kyoiku University, 1-5-1, Midorigaoka, Ikeda-shi, Osaka, 563-0026, Japan}
\email{anju.scorpion@icloud.com}
\allowdisplaybreaks
\begin{document}
\begin{abstract}
    Ohno--Wakabayashi's cyclic sum formula for multiple zeta-star values is generalized by Igarashi with one or two parameters.
    In this article, we give a possible answer for one of his problems about a generalization with three parameters.
\end{abstract}
\maketitle
\section{Introduction}
We sometimes encounter a remarkable symmetry on computing a cyclic sum on certain multiple nested sum.
For the case of the \emph{multiple zeta-star value}
\[\zeta^{\star}(k_{1},\ldots,k_{d})\coloneqq\sum_{0<m_{1}\le\cdots\le m_{d}}\prod_{i=1}^{d}\frac{1}{m_{i}^{k_{i}}},\]
defined for positive integers $k_{1},\ldots,k_{d}$ ($k_{d}\neq 1$), Ohno--Wakabayashi found the following formula.
\begin{theorem}[{\cite[Theorem 1]{ow06}}]\label{thm:ow}
    Let $k_{1},\ldots,k_{d}$ be positive integers and assume $k\coloneqq k_{1}+\cdots+k_{d}\neq d$.
    Then we have
    \[\sum_{i=1}^{d}\sum_{j=0}^{k_{i}-2}\zeta^{\star}(j+1,k_{i+1},\ldots,k_{d},k_{1},\ldots,k_{i-1},k_{i}-j)=k\zeta(k+1).\]
\end{theorem} 
Generalizing this theorem, Igarashi computed similar cyclic sums with two parameters. Define
\[Z_{I}^{\star}(k_{1},\ldots,k_{d};\alpha,\beta)\coloneqq\sum_{0\le m_{1}\le\cdots\le m_{d}}\frac{(\alpha)_{m_{1}}}{m_{1}!}\frac{m_{d}!}{(\alpha)_{m_{d}}}\frac{1}{(m_{1}+\beta)^{k_{1}}}\prod_{i=2}^{d}\frac{1}{(m_{i}+\alpha)(m_{i}+\beta)^{k_{i}-1}}\]
and
\[Z(a\mid b;\alpha,\beta)\coloneqq\sum_{n=0}^{\infty}\frac{1}{(n+\alpha)^{a}(n+\beta)^{b}},\]
for positive integers $d,k_{1},\ldots,k_{d}$ ($k_{d}\neq 1$), integers $a$ and $b$ ($a+b\ge 2$), complex numbers $\alpha,\beta$ satisfying $\Re(\alpha)>0$ and $\beta\notin\bbZ_{\le 0}$.
\begin{theorem}[{\cite[Theorem 1 (ii)]{igarashi22}}]\label{thm:igarashi}
    Let $k_{1},\ldots,k_{d}$ be positive integers and assume $k\coloneqq k_{1}+\cdots+k_{d}\neq d$.
    Then we have
    \begin{equation}\label{eq:igarashi}
        \sum_{i=1}^{d}\sum_{j=0}^{k_{i}-2}Z_{I}^{\star}(j+1,k_{i+1},\ldots,k_{d},k_{1},\ldots,k_{i-1},k_{i}-j)=(k-d)Z(d\mid k-d+1;\alpha,\beta)+dZ(d+1\mid k-d;\alpha,\beta).
    \end{equation}
\end{theorem} 
\begin{remark}
    \begin{enumerate}
        \item Theorem \ref{thm:igarashi} recovers Theorem \ref{thm:ow} by putting $\alpha=\beta=1$.
        \item Igarashi \cite[Theorem 1.1 (ii)]{igarashi11} obtained the special ($\alpha=\beta$) case of Theorem \ref{thm:igarashi} before he found it.
    \end{enumerate}
\end{remark}
In his paper \cite{igarashi22}, Igarashi proposed three problems. One of them \cite[(C3)]{igarashi22} asks us if there is a nice generalization of Theorem \ref{thm:igarashi} to the following series with three parameters:
\[Z_{I}^{\star}(k_{1},\ldots,k_{d};\alpha,\beta,\gamma)\coloneqq\sum_{0\le m_{1}\le\cdots\le m_{d}}\frac{(\alpha)_{m_{1}}(\beta)_{m_{1}}}{(\gamma)_{m_{1}}m_{1}!}\frac{(\gamma)_{m_{d}}m_{d}!}{(\alpha)_{m_{d}}(\beta)_{m_{d}}}\frac{1}{(m_{1}+\gamma)^{k_{1}}}\prod_{i=2}^{d}\frac{1}{(m_{i}+\alpha)(m_{i}+\beta)(m_{i}+\gamma)^{k_{i}-2}},\]
where $k_{1},\ldots,k_{d}$ are with the same condition as above and $\alpha,\beta$ and $\gamma$ are complex numbers satisfying $\Re(\alpha+\beta-\gamma)>0$ and $\alpha,\beta,\gamma\notin\bbZ_{\le 0}$.\\

The purpose of this article is to give a partially affirmative answer for it.
We use the following auxiliary sums:
\[Z_{II}^{\star}(k_{1},\ldots,k_{d};\alpha,\beta,\gamma)\coloneqq\sum_{0\le m_{0}\le m_{1}\le\cdots\le m_{d}}\frac{(\alpha)_{m_{0}}(\beta)_{m_{0}}}{(\gamma)_{m_{0}}m_{0}!}\frac{(\gamma)_{m_{d}}m_{d}!}{(\alpha)_{m_{d}}(\beta)_{m_{d}}}\prod_{i=1}^{d}\frac{1}{(m_{i}+\alpha)(m_{i}+\beta)(m_{i}+\gamma)^{k_{i}-2}}\]
and
\[Z(a\mid b\mid c;\alpha,\beta,\gamma)\coloneqq\sum_{n=0}^{\infty}\frac{1}{(n+\alpha)^{a}(n+\beta)^{b}(n+\gamma)^{c}},\]
for $\alpha,\beta,\gamma$ with the same condition, $k_{1},\ldots,k_{d}\in\bbZ_{\ge 2}$ and integers $a,b$ and $c$ satisfying $a+b+c\ge 2$.
\begin{theorem}[$=$ Theorem \ref{thm:main2}]\label{thm:main}
    Let $k_{1},\ldots,k_{d}$ be positive integers such that $k_{i}\ge 2$ for every $1\le i\le d$ and $\alpha,\beta,\gamma$ complex numbers satisfying $\Re(\alpha+\beta-\gamma)>0$ and $\alpha,\beta,\gamma\notin\bbZ_{\le 0}$.
    Then we have
    \begin{multline}
        \begin{split}
            \sum_{i=1}^{d}\sum_{j=0}^{k_{i}-3}Z_{I}^{\star}(j+1,k_{i+1},\ldots,k_{d},k_{1},\ldots,k_{i-1},k_{i}-j;\alpha,\beta,\gamma)
            +(\alpha+\beta-\gamma)\sum_{i=1}^{d}Z_{II}^{\star}(k_{i},\ldots,k_{d},k_{1},\ldots,k_{i-1},2;\alpha,\beta,\gamma)\\
            =dZ(d\mid d+1\mid k-2d;\alpha,\beta,\gamma)+dZ(d+1\mid d\mid k-2d;\alpha,\beta,\gamma)+(k-2d)Z(d\mid d\mid k-2d+1;\alpha,\beta,\gamma).
        \end{split}
    \end{multline}
\end{theorem}
\begin{remark}
    The condition of $(k_{1},\ldots,k_{d})$ in Theorem \ref{thm:main} is sometimes referred as \emph{having maximal height} in the context of multiple zeta values.
\end{remark}
Our main theorem can deduces a ``maximal height'' case of Igarashi's formula \eqref{eq:igarashi} as follows: put $\beta=\gamma$ in Theorem \ref{thm:main}.
Then the right-hand side becomes
\begin{multline}
    dZ(d\mid d+1\mid k-2d;\alpha,\beta,\beta)+dZ(d+1\mid d\mid k-2d;\alpha,\beta,\beta)+(k-2d)Z(d\mid d\mid k-2d+1;\alpha,\beta,\beta)\\
    =(k-d)Z(d\mid k-d+1;\alpha,\beta)+dZ(d+1\mid k-d;\alpha,\beta),
\end{multline}
and for $\bk=(k_{1},\ldots,k_{d})\in\bbZ_{\ge 2}^{d}$ and a positive integer $k$, we have
\[\alpha Z_{II}^{\star}(k+1,\bk;\alpha,\beta,\beta)=Z_{I}^{\star}(k,\bk;\alpha,\beta),\]
since the equality
\[\alpha\sum_{m=0}^{n}\frac{(\alpha)_{m}}{m!}=(\alpha+n)\frac{(\alpha)_{n}}{n!}\]
holds for $n\ge 0$.
\section*{Acknowledgements}
The second author is grateful to Takumi Maesaka for his helpful advices. This work is supported by Academic Research Club of KADOKAWA DWANGO Educational Institute.
\section{A proof}
\subsection{Computation on single sums}
Hereafter we fix a tuple $(\alpha,\beta,\gamma)$ with the condition in Theorem \ref{thm:main} and drop its information from notations appearing here. Denote $\alpha+\beta-\gamma$ by $t$.
We put
\begin{align}
    g_{m}^{\ep,\ep'}&\coloneqq\frac{(\alpha)_{m+\ep}(\beta)_{m+\ep'}}{(\gamma)_{m}m!}\\
\end{align}
for a non-negative integer $m$ and $\ep,\ep'\in\{0,1\}$.
We consider $g_{-1}^{\ep,\ep'}$ as $0$.
Note that reccurence relations
\begin{equation}\label{eq:rec1}
    \frac{g_{m}^{1,0}}{\alpha+m}=\frac{g_{m}^{0,1}}{\beta+m}=\frac{g_{m}^{1,1}}{(\alpha+m)(\beta+m)}=g_{m}^{0,0},
\end{equation}
\begin{equation}\label{eq:rec2}
    g_{m+1}^{\ep,\ep'}=\frac{(\alpha+m+\ep)(\beta+m+\ep')}{(\gamma+m)(m+1)}g_{m}^{\ep,\ep'}
\end{equation}
and
\begin{equation}\label{eq:rec3}
    g_{m+1}^{0,0}=\frac{g_{m}^{1,1}}{(m+\gamma)(m+1)}
\end{equation}
hold.
Moreover, we put
\[F_{m,n}\coloneqq g_{m}^{1,1}\sum_{n<l}\frac{1}{l-m}\frac{1}{g_{l}^{1,1}}\]
for non-negative integers $m$ and $n$ satisfying $m\le n$.
\begin{lemma}\label{lem:31}
For non-negative integers $m\le n$, we have
\[        
F_{m,n}=\sum_{h=0}^{m}g_{h}^{0,0}\left(\frac{1}{(n+1-h)g_{n+1}^{0,0}}-\sum_{n<l}\frac{t}{g_{l}^{1,1}}\right)
\]
\end{lemma}
\begin{proof}
By definition and \eqref{eq:rec1}, we have
\[\frac{1}{g_{m}^{1,1}}F_{m,n}
    =\sum_{n<l}\frac{1}{(l-m)(\alpha+l)(\beta+l)g_{l}^{0,0}}.\]
    The partial fraction decomposition
    \begin{equation}\label{eq:dec1}
        \frac{1}{(l-m)(\alpha+l)(\beta+l)}=\frac{1}{(\alpha+m)(\beta+m)}\left(\frac{1}{l-m}-\frac{\alpha+\beta+m+l}{(\alpha+l)(\beta+l)}\right),
    \end{equation}
    deduces
    \[\frac{(\alpha+m)(\beta+m)}{g_{m}^{1,1}}F_{m,n}
        =\sum_{n<l}\frac{1}{l-m}\frac{1}{g_{l}^{0,0}}-\sum_{n<l}\frac{\alpha+\beta+m+l}{(\alpha+l)(\beta+l)g_{l}^{0,0}}\]
        and thus \eqref{eq:rec1} shows 
        \begin{align}
            \begin{split}\label{eq:31_key}
            \frac{(\alpha+m)(\beta+m)}{g_{m}^{1,1}}F_{m,n}+\sum_{n<l}\frac{\alpha+\beta+m+l}{g_{l}^{1,1}}
        &=\sum_{n<l}\frac{1}{l-m}\frac{1}{g_{l}^{0,0}}\\
        &=\frac{1}{(n+1-m)g_{n+1}^{0,0}}+\sum_{n<l}\frac{1}{l+1-m}\frac{1}{g_{l+1}^{0,0}}
            \end{split}
        \end{align}
        Denote by $A$ the second sum on the right-hand side.
        Using the reccurence relation \eqref{eq:rec3}, we obtain
        \[
            A=\sum_{n<l}\frac{(l+\gamma)(l+1)}{l+1-m}\frac{1}{g_{l}^{1,1}}
        \]
        Then we can apply the decomposition
        \begin{equation}\label{eq:dec2}
            \frac{(l+\gamma)(l+1)}{l+1-m}=l+\gamma+m+\frac{(\gamma+m-1)m}{l+1-m}
        \end{equation}
        and get
        \[A=\sum_{n<l}\left(l+\gamma+m+\frac{(\gamma+m-1)m}{l+1-m}\right)\frac{1}{g_{l}^{1,1}}=\sum_{n\le l}\frac{l+\gamma+m}{g_{l}^{1,1}}+\frac{(\gamma+m-1)m}{g_{m-1}^{1,1}}F_{m-1,n}.\]
        Combining this and \eqref{eq:31_key}, it follows that
        \[\frac{(\alpha+m)(\beta+m)}{g_{m}^{1,1}}F_{m,n}-\frac{(\gamma+m-1)m}{g_{m-1}^{1,1}}F_{m-1,n}=\frac{1}{(n+1-m)g_{n+1}^{0,0}}-\sum_{n<l}\frac{t}{g_{l}^{1,1}}\]
        holds and the left-hand side is equal to $(F_{m,n}-F_{m-1,n})/g_{m}^{0,0}$ from the reccurence equations \eqref{eq:rec1} and \eqref{eq:rec3}.
        Therefore, multiplying $g_{m}^{0,0}$ on both sides and then summing them up over $m$ completes the proof.
\end{proof}
\begin{lemma}\label{lem:32}
    For non-negative integers $m$ and $n$, if $m\le n$ we have
    \[F_{m,n}-(\alpha+\beta+m+n)\frac{g_{m}^{0,0}}{g_{n}^{1,1}}=\sum_{h=0}^{m-1}\frac{1}{n-h}\frac{g_{h}^{0,0}}{g_{n}^{0,0}}-t\sum_{h=0}^{m}g_{h}^{0,0}\sum_{n\le l}\frac{1}{g_{l}^{1,1}}.\]
\end{lemma}
\begin{proof}
    From \eqref{eq:31_key}, \eqref{eq:rec1} and \eqref{eq:dec1}, we see that
    \begin{align}
        F_{m,n}
        &=g_{m}^{1,1}\sum_{n<l}\frac{1}{l-m}\frac{1}{g_{l}^{1,1}}\\
        &=g_{m}^{1,1}\sum_{n<l}\frac{1}{(l-m)(\alpha+l)(\beta+l)}\frac{1}{g_{l}^{0,0}}\\
        &=g_{m}^{0,0}\sum_{n<l}\left(\frac{1}{l-m}-\frac{\alpha+\beta+m+l}{(\alpha+l)(\beta+l)}\right)\frac{1}{g_{l}^{0,0}}
    \end{align}
    holds.
    Since the first sum on the right-hand side is computed by \eqref{eq:rec3} as
    \begin{align}
        \sum_{n<l}\frac{1}{l-m}\frac{1}{g_{l}^{0,0}}
        &=\sum_{n<l}\frac{l(l+\gamma-1)}{l-m}\frac{1}{g_{l-1}^{1,1}}\\
        &=\sum_{n<l}\left(\frac{m(\gamma+m-1)}{l-m}+l+\gamma+m-1\right)\frac{1}{g_{l-1}^{1,1}}
    \end{align}
    where we used \eqref{eq:dec2} in the second equality, we have
    \begin{align}
        \begin{split}\label{eq:32_key}
        F_{m,n}
        &=g_{m}^{0,0}\sum_{n<l}\left(\frac{m(\gamma+m-1)}{l-m}+l+\gamma+m-1\right)\frac{1}{g_{l-1}^{1,1}}-g_{m}^{0,0}\sum_{n<l}\frac{\alpha+\beta+m+l}{(\alpha+l)(\beta+l)}\frac{1}{g_{l}^{0,0}}\\
        &=g_{m-1}^{1,1}\sum_{n<l}\frac{1}{l-m}\frac{1}{g_{l-1}^{1,1}}+g_{m}^{0,0}\sum_{n\le l}\frac{l+\gamma+m}{g_{l}^{1,1}}-g_{m}^{0,0}\sum_{n<l}\frac{\alpha+\beta+m+l}{g_{l}^{1,1}}\\
        &=F_{m-1,n}+\frac{1}{n-m+1}\frac{g_{m-1}^{1,1}}{g_{n}^{1,1}}+(\alpha+\beta+m+n)\frac{g_{m}^{0,0}}{g_{n}^{1,1}}-tg_{m}^{0,0}\sum_{n\le l}\frac{1}{g_{l}^{1,1}}.
        \end{split}
    \end{align}
    Then we use the decomposition \eqref{eq:dec1} on the second term, namely,
    \begin{align}
        \frac{1}{n-m+1}\frac{g_{m-1}^{1,1}}{g_{n}^{1,1}}
        &=\frac{1}{(n-m+1)(\alpha+n)(\beta+n)}\frac{g_{m-1}^{1,1}}{g_{n}^{0,0}}\\
        &=\frac{1}{(\alpha+m-1)(\beta+m-1)}\left(\frac{1}{n-m+1}-\frac{\alpha+\beta+m+n-1}{(\alpha+n)(\beta+n)}\right)\frac{g_{m-1}^{1,1}}{g_{n}^{0,0}}\\
        &=\frac{1}{n-m+1}\frac{g_{m-1}^{0,0}}{g_{n}^{0,0}}-(\alpha+\beta+m+n-1)\frac{g_{m-1}^{1,1}}{g_{n}^{1,1}}.
    \end{align}
    Plugging this calculation in \eqref{eq:32_key}, we can make telescopic terms as
    \[
        F_{m,n}-F_{m-1,n}+(\alpha+\beta+m+n-1)\frac{g_{m-1}^{1,1}}{g_{n}^{1,1}}-(\alpha+\beta+m+n)\frac{g_{m}^{0,0}}{g_{n}^{1,1}}
        =\frac{1}{n-m+1}\frac{g_{m-1}^{0,0}}{g_{n}^{0,0}}-tg_{m}^{0,0}\sum_{n\le l}\frac{1}{g_{l}^{1,1}}.
    \]
    By taking the sum on both sides, it is shown that
    \[
        F_{m,n}-(\alpha+\beta+m+n)\frac{g_{m}^{0,0}}{g_{n}^{1,1}}=\sum_{h=0}^{m}\left(\frac{1}{n-h+1}\frac{g_{h-1}^{0,0}}{g_{n}^{0,0}}-tg_{h}^{0,0}\sum_{n\le l}\frac{1}{g_{l}^{1,1}}\right)
    \]
    holds. This deduces the desired result by shifting the variable on the first sum on the right-hand side and using $g_{-1}^{0,0}=0$.
\end{proof}
\subsection{Connected sum of Hoffman--Ohno type}
Put $X\coloneqq\{\le, =\}$ and denote by $X^{\bullet}$ the free monoid generated by $X$.
\begin{definition}
    Let $\circ_{1},\ldots,\circ_{d}$ ($d\ge 1$) be elements of $X$ and $s=(\circ_{1},\ldots,\circ_{d})$ be a non-empty word of $X^{\bullet}$.
    For integers $1\le i<j\le d+1$, we put
    \[S(s)\coloneqq\{(m_{1},\ldots,m_{d+1})\in\bbZ_{\ge 0}^{d+1}\mid m_{1}\circ_{1}\cdots\circ_{d}m_{d+1}\}\]
    and
    \[S_{i,j}(s)\coloneqq\{(m_{1},\ldots,m_{d+1})\in S(s)\mid m_{i}\neq m_{j}\}.\]
\end{definition}
The following lemma is obvious but useful assertions to prove the main theorem.
\begin{lemma}\label{lem:set}
    For an integer $d\ge 2$, we have the following.
    \begin{enumerate}
        \item\label{set1} $S_{1,d}(\le^{d-1})=S_{1,d-1}(\le^{d-1})\sqcup S_{d-1,d}(=^{d-2},\le)$.
        \item\label{set2} $S_{1,d}(\le^{d-1})=S_{2,d}(\le^{d-1})\sqcup S_{1,2}(\le,=^{d-2})$.
        \item\label{set3} $S_{1,d}(\le^{d-1})\sqcup S(=^{d-1})=S(\le^{d-1})$.
    \end{enumerate}
\end{lemma}
Moreover, for a positive integer $d$, $\bm=(m_{1},\ldots,m_{d})\in\bbZ_{\ge 0}^{d}$ and $\bk=(k_{1},\ldots,k_{d})\in\bbZ_{\ge 1}^{d}$, we define
\[\Pi\binom{\bm}{\bk}\coloneqq\prod_{i=1}^{d}\frac{1}{(m_{i}+\alpha)(m_{i}+\beta)(m_{i}+\gamma)^{k_{i}-2}}.\]
Note that $\Pi\binom{m}{2}g_{m}^{1,1}=g_{m}^{0,0}$ holds for every $m\ge 0$.
\begin{remark}
    With these notations, the definition of the series $Z_{I}^{\star}$ is written as
    \[Z_{I}^{\star}(k,\bk)=\sum_{(l,\bm,l')\in S(\le^{d})}\frac{g_{l}^{0,0}}{g_{l'}^{0,0}}\frac{1}{(l+\gamma)^{k}}\Pi\binom{\bm,l'}{\bk}\]
    for $k\ge 1$ and $\bk\in(\bbZ_{\ge 1}^{d-1}\times\bbZ_{\ge 2})$.
\end{remark}
\begin{definition}
    For positive integers $k_{1},\ldots,k_{d}$ ($d\ge 2$) for some $k_{i}\ge 2$, we define
    \[H(k_{1},\ldots,k_{d})\coloneqq\sum_{(m_{1},\bm,m_{d})\in S_{1,d}(\le^{d-1})}\frac{g_{m_{1}}^{0,0}}{g_{m_{d}}^{0,0}}\frac{1}{(m_{1}+\gamma)^{k_{1}-1}}\Pi\binom{\bm,m_{d}}{k_{2},\ldots,k_{d}}\frac{1}{m_{d}-m_{1}}.\]
\end{definition}
\begin{remark}
    By definition, the equality
    \[H(k_{1},\ldots,k_{d})\coloneqq\sum_{(m_{1},\bm,m_{d})\in S_{1,d}(\le^{d-1})}\frac{g_{m_{1}}^{1,1}}{g_{m_{d}}^{0,0}}\Pi\binom{m_{1},\bm,m_{d}}{1+k_{1},k_{2},\ldots,k_{d}}\frac{1}{m_{d}-m_{1}}.\]
    holds.
\end{remark}
\subsection{Transport relations}
\begin{proposition}\label{prop:trans1}
    For non-negative integers $d$ and $j$, a positive integer $k\ge 2$ and $\bk\in\bbZ_{\ge 1}^{d}$, if $0\le j\le k-2$ we have
    \begin{equation}\label{eq:trans1}
        H(j+1,\bk,k-j)=H(j+2,\bk,k-j-1)-Z_{I}^{\star}(j+1,\bk,k-j)+Z(d+1\mid d+1\mid w-2d-1),
    \end{equation}
    where $w$ is the sum of all components of $\bk$ and $k$.
\end{proposition}
\begin{proof}
    By definition, we have
    \[H(j+1,\bk,k-j)=\sum_{(l,\bm,l')\in S_{1,d+2}(\le^{d+1})}\frac{g_{l}^{0,0}}{g_{l'}^{0,0}}\frac{1}{l'-l}\frac{1}{(l+\gamma)^{j}}\Pi\binom{\bm,l'}{\bk,k-j}.\]
    Using the partial fraction decomposition
    \[\frac{1}{l'-l}\frac{1}{l'+\gamma}=\left(\frac{1}{l'-l}-\frac{1}{l'+\gamma}\right)\frac{1}{l+\gamma},\]
    we obtain
    \begin{align}
        H(j+1,\bk,k-j)
        &=\sum_{(l,\bm,l')\in S_{1,d+2}(\le^{d+1})}\frac{g_{l}^{0,0}}{g_{l'}^{0,0}}\frac{1}{(l+\gamma)^{j+1}}\left(\frac{1}{l'-l}-\frac{1}{l'+\gamma}\right)\Pi\binom{\bm,l'}{\bk,k-j-1}\\
        &=H(j+2,\bk,k-j-1)-Z_{I}^{\star}(j+1,\bk,k-j)+Z(d+1\mid d+1\mid w-2d-1),
    \end{align}
    where we used Lemma \ref{lem:set} \eqref{set3} in the last equality.
\end{proof}
\begin{corollary}\label{cor:trans2}
    Let $d,k,\bk$ and $w$ be with the same conditions as Proposition \ref{prop:trans1}.
    Then we have
    \begin{equation}\label{eq:trans2}
        H(1,\bk,k)=H(k-1,\bk,2)-\sum_{j=0}^{k-3}Z_{I}^{\star}(j+1,\bk,k-j)+(k-2)Z(d+1;d+1;w-2d-1).
    \end{equation}
\end{corollary}
\begin{proof}
    We get this claim by summing up \eqref{eq:trans1} over $j=0,\ldots,k-3$.
\end{proof}
\begin{proposition}\label{prop:trans3}
    Let $d,k,\bk$ and $w$ be with the same conditions as Proposition \ref{prop:trans1}.
    Then we have
    \begin{equation}\label{eq:trans3}
        H(k-1,\bk,2)=H(1,k,\bk)-tZ_{II}^{\star}(k,\bk,2)+Z(d+1\mid d+2\mid w-2d-2)+Z(d+2\mid d+1\mid w-2d-2).
    \end{equation}
\end{proposition}
\begin{proof}
    The definition of $H$ and \eqref{eq:rec1} yield
    \[H(k-1,\bk,2)=\sum_{(l,\bm,m,l')\in S_{1,d+2}(\le^{d+1})}\frac{g_{l}^{1,1}}{g_{l'}^{1,1}}\frac{1}{l'-l}\Pi\binom{l,\bm,m}{k,\bk}.\]
    Using Lemma \ref{lem:set} \eqref{set1}, we have
    \begin{align}
        \begin{split}\label{eq:33_key}
        H(k-1,\bk,2)
        &=\sum_{(l,\bm,m)\in S_{1,d+1}(\le^{d})}\Pi\binom{l,\bm,m}{k,\bk}\cdot g_{l}^{1,1}\sum_{m\le l'}\frac{1}{l'-l}\frac{1}{g_{l'}^{1,1}}+\sum_{l\ge 0}\Pi\binom{\overbrace{l,\ldots,l}^{d+1}}{k,\bk}\cdot g_{l}^{1,1}\sum_{l<l'}\frac{1}{l'-l}\frac{1}{g_{l'}^{1,1}}\\
        &=\sum_{(l,\bm,m)\in S_{1,d+1}(\le^{d})}\Pi\binom{l,\bm,m}{k,\bk}F_{l,m-1}+\sum_{l\ge 0}\Pi\binom{\overbrace{l,\ldots,l}^{d+1}}{k,\bk}F_{l,l}.
    \end{split}
    \end{align}
Let us calculate the sums on the right-hand side.
For the first term, by Lemma \ref{lem:31} we obtain    
\begin{equation}\label{eq:after31}
    \sum_{(l,\bm,m)\in S_{1,d+1}(\le^{d})}\Pi\binom{l,\bm,m}{k,\bk}F_{l,m-1}
    =\sum_{(l,\bm,m)\in S_{1,d+1}(\le^{d})}\Pi\binom{l,\bm,m}{k,\bk}\sum_{0\le h\le l}g_{h}^{0,0}\left(\frac{1}{m-h}\frac{1}{g_{m}^{0,0}}-\sum_{m\le l'}\frac{t}{g_{l'}^{1,1}}\right).
\end{equation}
On the other hand, using Lemma \ref{lem:32} we compute the second term as
\begin{align}
    \begin{split}\label{eq:after32}
    \sum_{l\ge 0}\Pi\binom{\overbrace{l,\ldots,l}^{d+1}}{k,\bk}F_{l,l}
    &=\sum_{l\ge 0}\Pi\binom{\overbrace{l,\ldots,l}^{d+1}}{k,\bk}\left((\alpha+\beta+2l)\frac{g_{l}^{0,0}}{g_{l}^{1,1}}+\sum_{0\le h<l}\frac{1}{l-h}\frac{g_{h}^{0,0}}{g_{l}^{0,0}}-t\sum_{0\le h\le l}g_{h}^{0,0}\sum_{l\le l'}\frac{1}{g_{l'}^{1,1}}\right)\\
    &=\sum_{l\ge 0}\Pi\binom{\overbrace{l,\ldots,l}^{d+1}}{k,\bk}\left(\frac{1}{l+\alpha}+\frac{1}{l+\beta}+\sum_{0\le h<l}\frac{1}{l-h}\frac{g_{h}^{0,0}}{g_{l}^{0,0}}-t\sum_{0\le h\le l}g_{h}^{0,0}\sum_{l\le l'}\frac{1}{g_{l'}^{1,1}}\right).
    \end{split}
\end{align}
Combining \eqref{eq:after31}, \eqref{eq:after32} and \eqref{eq:33_key}, it follows that
\begin{align}
    H(k-1,\bk,2)
        &=\sum_{(l,\bm,m)\in S_{1,d+1}(\le^{d})}\Pi\binom{l,\bm,m}{k,\bk}F_{l,m-1}+\sum_{l\ge 0}\Pi\binom{\overbrace{l,\ldots,l}^{d+1}}{k,\bk}F_{l,l}\\
        &=\begin{multlined}[t]
            \sum_{(l,\bm,m)\in S_{1,d+1}(\le^{d})}\Pi\binom{l,\bm,m}{k,\bk}\left(\sum_{0\le h\le l}g_{h}^{0,0}\left(\frac{1}{m-h}\frac{1}{g_{m}^{0,0}}-\sum_{m\le l'}\frac{t}{g_{l'}^{1,1}}\right)\right)\\
            +\sum_{l\ge 0}\Pi\binom{\overbrace{l,\ldots,l}^{d+1}}{k,\bk}\left(\frac{1}{l+\alpha}+\frac{1}{l+\beta}+\sum_{0\le h<l}\frac{1}{l-h}\frac{g_{h}^{0,0}}{g_{l}^{0,0}}-t\sum_{0\le h\le l}g_{h}^{0,0}\sum_{l\le l'}\frac{1}{g_{l'}^{1,1}}\right).
        \end{multlined}\\
        &=\begin{multlined}[t]
            \sum_{(h,l,\bm,m)\in S_{2,d+2}(\le^{d+1})}\frac{g_{k}^{0,0}}{g_{m}^{0,0}}\frac{1}{m-h}\Pi\binom{l,\bm,m}{k,\bk}
        -t\sum_{(h,l,\bm',l')\in S_{2,d+2}(\le^{d+2})}\frac{g_{h}^{0,0}}{g_{l'}^{1,1}}\Pi\binom{l,\bm'}{k,\bk}\\
        +Z(d+1\mid d+2\mid w-2d-2)+Z(d+2\mid d+1\mid w-2d-2)\\
        +\sum_{(h,l,\ldots,l)\in S_{1,2}(\le,=^{d})}\frac{g_{h}^{0,0}}{g_{l}^{0,0}}\frac{1}{l-h}\Pi\binom{\overbrace{l,\ldots,l}^{d+1}}{k,\bk}
        -t\sum_{(h,l,\ldots,l,l')\in S(\le,=^{d},\le)}\frac{g_{h}^{0,0}}{g_{l'}^{1,1}}\Pi\binom{\overbrace{l,\ldots,l}^{d+1}}{k,\bk}
        \end{multlined}
\end{align}
holds. Finally, using Lemma \ref{lem:set} \eqref{set1} (resp.~\eqref{set2}) for the first and fifth (resp.~second and sixth) terms, we obtain
\[H(k-1,\bk,2)=H(1,k,\bk)-tZ_{II}^{\star}(k,\bk,2)+Z(d+1\mid d+2\mid w-2d-2)+Z(d+2\mid d+1\mid w-2d-2).\]
\end{proof}
\begin{theorem}\label{thm:main2}
    Let $k_{1},\ldots,k_{d}$ be positive integers such that $k_{i}\ge 2$ for $1\le i\le d$.
    Then we have
    \begin{multline}
            \sum_{i=1}^{d}\sum_{j=0}^{k_{i}-3}Z_{I}^{\star}(j+1,k_{i+1},\ldots,k_{d},k_{1},\ldots,k_{i-1},k_{i}-j)
            +t\sum_{i=1}^{d}Z_{II}^{\star}(k_{i},\ldots,k_{d},k_{1},\ldots,k_{i-1},2)\\
            =dZ(d\mid d+1\mid k-2d)+dZ(d+1\mid d\mid k-2d)+(k-2d)Z(d\mid d\mid k-2d+1),
    \end{multline}
    where $k\coloneqq k_{1}+\cdots+k_{d}$.
\end{theorem}
\begin{proof}
    For $k'\ge 1$ and $\bk\in\bbZ_{\ge 2}^{d'}$, using Corollary \ref{cor:trans2} and Proposition \ref{prop:trans3} we have
\begin{multline}
H(1,\bk,k')-H(1,k',\bk)=-\sum_{j=0}^{k'-3}Z_{I}^{\star}(j+1,\bk,k'-j)-tZ_{II}^{\star}(k',\bk,2)\\
+Z(d'+1\mid d'+2\mid w-2d'-2)+Z(d'+2\mid d'+1\mid w-2d'-2)+(k'-2)Z(d'+1;d'+1;w-2d'-1),
\end{multline}
where $w$ is the sum of every entry of $\bk$ and $k'$.
Then, making repeated use of this formula, we obtain the theorem.
\end{proof}

\end{document}